\documentclass[12pt]{amsart}
\usepackage[a4paper, total={16cm,24cm}]{geometry}
\usepackage[T1]{fontenc}
\usepackage[utf8]{inputenc}
\usepackage[english]{babel}

\usepackage[table]{xcolor}

\usepackage{pgfplots}
\pgfplotsset{width=10cm,compat=1.9}

\usepackage{amsfonts,amsmath,amsthm,amssymb}
\usepackage{enumitem} 
\usepackage{siunitx} 
\usepackage{color, graphicx}
\usepackage{tikz}
\usepackage{url}

\providecommand{\U}[1]{\protect\rule{.1in}{.1in}}

\newtheorem{theorem}{Theorem}[section]

\newtheorem{corollary}[theorem]{Corollary}

\newtheorem{definition}[theorem]{Definition}

\newtheorem{notation}[theorem]{Notation}

\newtheorem{proposition}[theorem]{Proposition}

\def\N{\mathbb{N}}

\title{A verification of Wilf's conjecture up to genus 100}

\thanks{The first author was partially supported by CMUP, a member of LASI, which is financed by national funds through FCT – Fundação para a Ciência e a Tecnologia, I.P., under the projects with reference UIDB/00144/2020 and UIDP/00144/2020.\\
  The authors gratefully acknowledge support from ULCO for visits of the first author to Calais in May 2019 and May 2023.}

\author{M. Delgado, S. Eliahou and J. Fromentin}

\date{\today}

\begin{document}
\keywords{Numerical semigroup; genus}

\subjclass[2010]{20M14}

\begin{abstract} For a numerical semigroup $S \subseteq \mathbb{N}$, let $m,e,c,g$ denote its multiplicity, embedding dimension, conductor and genus, respectively. Wilf's conjecture (1978) states that $e(c-g) \ge c$. As of 2023, Wilf's conjecture has been verified by computer up to genus $g \le 66$. In this paper, we extend the verification of Wilf's conjecture up to genus $g \le 100$. This is achieved by combining three main ingredients: (1) a theorem in 2020 settling Wilf's conjecture in the case $e \ge m/3$, (2) an efficient trimming of the tree $\mathcal{T}$ of numerical groups identifying and cutting out irrelevant subtrees, and (3) the implementation of a fast parallelized algorithm to construct the tree $\mathcal{T}$ up to a given genus. We further push the verification of Wilf's conjecture up to genus $120$ in the particular case where $m$ divides $c$. Finally, we unlock three previously unknown values of the number $n_g$ of numerical semigroups of genus $g$, namely for $g=73,74,75$.  
\end{abstract}

\maketitle

\section{Introduction}\label{sec introduction}

A \emph{numerical semigroup} is a cofinite submonoid $S$ of $\mathbb{N}$, i.e. a subset containing $0$, stable under addition and with finite complement $\mathbb{N} \setminus S$. Equivalently, it is a set of the form $S=\langle a_1,\dots,a_n \rangle = \mathbb{N}a_1+\cdots+\mathbb{N}a_n$ where $a_1,\dots,a_n$ are positive integers with $\gcd(a_1,\dots,a_n)=1$, called \emph{generators} of $S$. The least such $n$ is usually denoted $e$ and called the \emph{embedding dimension} of $S$, see below.

Let $S$ be a numerical semigroup and $S^*=S \setminus \{0\}$. A \emph{primitive element} of $S$ is an element $a \in S^*\setminus (S^*+S^*)$, i.e. a nonzero element of $S$ which is not the sum of two nonzero elements of $S$. Let $P=P(S)$ denote the set of primitive elements of $S$. It is easy to see that $P$ is finite and is the unique minimal generating set of $S$. The \emph{embedding dimension} of $S$ is $e=e(S)=|P|$, the \emph{multiplicity} of $S$ is $m=m(S)=\min S^*$, the \emph{Frobenius number} of $S$ is $F=F(S)=\max(\mathbb{Z} \setminus S)$ and the \emph{conductor} of $S$ is $c=c(S)=F+1$, satisfying $c+\mathbb{N} \subseteq S$ and minimal with respect to that property. The \emph{genus} of $S$ is $g = g(S)=|\mathbb{N} \setminus S|$, its number of gaps. We partition $S$ as
$$
S = L \sqcup R,
$$
where $L = L(S)=\{a \in S \mid a < F(S)\}$ and $R = R(S)= \{a \in S \mid a > F(S)\}$, the \textit{left part} and \textit{right part} of $S$, respectively. 

Wilf's conjecture is the inequality 
$$e|L| \ge c.$$

While this conjecture remains open since 1978,
it has been verified in several cases. For convenience, we list some of them in a single statement together with their respective references.

\begin{theorem}\label{thm known cases}
Wilf's conjecture holds for all numerical semigroups satisfying one of the following conditions:
\begin{enumerate}
\item $e \le 3$~\cite{FroebergGottliebHaeggkvist1987SF-numerical}
\item $c \le 3m$~\cite{Eliahou2018JEMS-Wilfs} 
\item $e \ge m/3$~\cite{Eliahou2020EJC-graph}
\item $|L| \le 12$~\cite{EliahouMarin-Aragon2021CiA-numerical}
\item $m \le 19$~\cite{KliemStump2022DCG-new}
\end{enumerate}
\end{theorem}
See also~\cite{Delgado2020-survey} for an extensive recent survey of partial results on Wilf's conjecture, and \cite{Bras-Amoros2008SF-Fibonacci, BrunsGarcia-SanchezONeillWilburne2020IJAC-Wilfs, Delgado2018MZ-question, DobbsMatthews2006, EliahouFromentin2019SF-misses, FromentinHivert2016MC-Exploring, Kaplan2012JPAA-Counting, MoscarielloSammartano2015MZ-conjecture, Sammartano2012SF-Numerical, Selmer1977JRAM-linear, Sylvester1882AJM-Subvariants, Wilf1978AMM-circle} for some other relevant papers.

\smallskip
The first significant verification by computer of Wilf's conjecture up to a given genus was accomplished in 2008 by Bras-Amor\'os~\cite{Bras-Amoros2008SF-Fibonacci}. There, Wilf's conjecture was announced to hold for all numerical semigroups of genus $g \le 50$. This result was extended to genus $g \le 60$ in 2016 by Fromentin-Hivert~\cite{FromentinHivert2016MC-Exploring}, to genus $g \le 65$ in 2021 by Bras-Amor\'os and Rodr{\'\i}guez~\cite{Bras-AmorosRodriguez2021-New}, and finally to genus $g \le 66$ in 2023 by Bras-Amor\'os~\cite{Bras-Amoros2023pp-seeds}.

In this paper, we extend the verification of Wilf's conjecture up to genus $g \le 100$. See Theorem~\ref{thm wilf 100} below. This is achieved by combining three main ingredients:

\begin{enumerate}
\item The main result of~\cite{Eliahou2020EJC-graph} stating that Wilf's conjecture holds in case $e \ge m/3$.
\item The method from~\cite{Delgado2019ae-Trimming} to exploit the above result by efficiently trimming the tree $\mathcal{T}$ of numerical semigroups (see below) and thereby drastically reduce the number of numerical semigroups up to a given large genus to test for Wilf's conjecture.
\item A fast parallelisable algorithm to enumerate all numerical semigroups up to a given large genus~\cite{FromentinHivert2016MC-Exploring}.
\end{enumerate}
For general reference books on numerical semigroups, see~\cite{Ramirez-Alfonsin2005Book-Diophantine, RosalesGarcia2009Book-Numerical}.

\section{The tree $\mathcal{T}$}\label{sec tree}

The set of all numerical semigroups can be organised in a rooted tree $\mathcal{T}$, with root $\mathbb{N}$ of genus $0$, such that all numerical semigroups of genus $g$ lie at distance $g$ from the root. Before recalling its construction, we introduce some terminology.

\subsection{Left and right primitive elements}
Let $S$ be a numerical semigroup. Recall the above partition $S = L \sqcup R$ relative to the Frobenius number $F(S)$. Accordingly, we partition the set $P$ of primitive elements of $S$ as 
$$
P = (P \cap L) \sqcup (P \cap R).
$$
We call \emph{left primitive} the elements of $P \cap L$ and \emph{right primitive} those of $P \cap R$. We denote
$$
\begin{aligned}
e &= |P|, \\
e_l &= |P \cap L|, \\
e_r &= |P \cap R|.
\end{aligned}
$$
Thus
$$
e = e_l + e_r.
$$

\subsection{The children of $S$}
We now briefly recall the construction of the tree $\mathcal{T}$. Let $S$ be a numerical semigroup, and let $a \in P \cap R$ be a right primitive element of $S$, if any. Then the set $S'=S\setminus \{a\}$ is still a numerical semigroup, as easily seen.

A \emph{child} of $S$ in $\mathcal{T}$ is any numerical semigroup of the above form
$$S'=S \setminus \{a\}$$
where $a \in P \cap R$. Clearly, the number of children of $S$ in $\mathcal{T}$ is equal to $e_r$.

\medskip
The Frobenius number, genus and multiplicity of a child $S'=S\setminus \{a\}$ of $S$ with $a \in P \cap R$ are easy to determine. Indeed, one has
$$
\left\{
\begin{aligned}
F(S') & = a, \\
g(S') & = g(S)+1
\end{aligned}
\right.
$$
as easily seen. As for the multiplicity of $S'$, one has
$$
m(S')=m(S)
$$
whenever $c(S) > m(S)$, see Proposition~\ref{prop child}. The only numerical semigroups $S$ for which $c(S)=m(S)$ are the so-called \textit{ordinary} or \textit{superficial} numerical semigroups $O_m$, defined for any $m \ge 1$ by
$$O_m=\langle m,m+1,\dots,2m-1 \rangle=\{0\} \sqcup (m+\mathbb{N}).$$
Then $O_m$ is of multiplicity $m$ and conductor $c=m$. It has exactly $m$ children, namely $O_m \setminus \{m+i\}$ for $0 \le i \le m-1$. They still have multiplicity $m$ for $1 \le i \le m-1$. The only case where the multiplicity differs is at $i=0$, for which $O_m \setminus \{m\}=O_{m+1}$ is of multiplicity $m+1$.

\smallskip
Note finally that any numerical semigroup of multiplicity $m \ge 1$ is a descendant of $O_m$ in the tree $\mathcal{T}$.

\subsection{Numerical semigroups of given genus} 

 For $g \in \N$, let $n_g$ denote the number of numerical semigroups of genus $g$. It is well known that $n_g$ is finite. The first few values of $n_g$ are $$(n_0,n_1,n_2,n_3,n_4,n_5,n_6)=(1,1,2,4,7,12,23).$$ In her famous paper~\cite{Bras-Amoros2008SF-Fibonacci}, Maria Bras-Amor\'{o}s conjectured that $n_g$ behaves like the $g$th Fibonacci number $F_g$, with a growth rate tending to the golden ratio $\phi=\frac{1+\sqrt5}2\approx 1.618$ and satisfying \begin{equation}\label{conj BA}
     n_g \ge n_{g-1}+n_{g-2}
 \end{equation} for all $g \ge 2$. The conjectured growth rate of $n_g$ was subsequently confirmed by A.~Zhai~\cite{Zhai}, implying $n_g \ge n_{g-1}$ for $g$ large enough. Yet the conjectured inequality~\eqref{conj BA} remains widely open to this day, as is the case for the much weaker inequality $n_g \ge n_{g-1}$ for all $g \ge 1$.
 
 Until now, the exact value of $n_g$ had been computed up to $g \le 72$. In this paper, with massive computations using a distributed version of the fast algorithms in~\cite{FromentinHivert2016MC-Exploring}, we unlock three new values, namely 
 $$
 \begin{array}{rcr}
 n_{73} &=& 6\,832\,823\,876\,813\,577, \\
 n_{74}&=& 11\,067\,092\,660\,179\,522,\\ 
 n_{75} &=& 17\,924\,213\,336\,425\,401.
 \end{array}
 $$
 See Section~\ref{sec experiments} for more details.

\section{Trimming $\mathcal{T}$}\label{sec trim}

The main ideas proposed in~\cite{Delgado2019ae-Trimming} to trim the tree $\mathcal{T}$ so as to drastically reduce the verification of Wilf's conjecture up to a given genus $G$ are the following ones.

\begin{notation} Let $G \ge 1$. We denote by $\mathcal{T}_G$ the subtree of $\mathcal{T}$ consisting of all numerical semigroups $S$ of genus $g(S) \le G$.  
\end{notation}

\begin{proposition}\label{prop child} Let $S \neq O_m$ be a numerical semigroup of multiplicity $m$. Let $S'$ be a child of $S$. Then 
\begin{enumerate}
\item $m(S')=m(S)$ \vspace{1mm}
\item $e_l(S') \ge e_l(S)$ \vspace{1mm}
\item $e(S)-1\le e(S') \le e(S)$
\end{enumerate}
\end{proposition}

\begin{proof} We have $S' = S \setminus \{a\}$ for some right primitive element $a \in P \cap R$. Denote $m,c,P,L$ the multiplicity, conductor, primitive elements and left part of $S$, respectively, and $m',c',P',L'$ the corresponding objects for $S'$.

Since $a \ge c > m$, it follows that $\min (S')^* = \min S^*$, i.e. $m'=m$. Moreover, since $c'=a+1>c$, it follows that $L \subset L'$, and any left primitive element $a \in P \cap L$ of $S$ remains left primitive in $S'$. Thus $P \cap L \subseteq P' \cap L'$, implying $e_l(S') \ge e_l(S)$. Finally, as easily seen, one has either $P'= P \setminus \{a\}$ or $P' = P \setminus \{a\} \sqcup \{a+m\}$. The former occurs when $a+m=s_1+s_2$ for some pair $\{s_1,s_2\} \subset S^*$ distinct from $\{a,m\}$, whereas the latter occurs when $a+m$ has no other representation as an element of $S^*+S^*$. Therefore $e(S)-1\le e(S') \le e(S)$.
\end{proof}

\begin{corollary}\label{cor el}
Let $S \neq O_m$ be a numerical semigroup of multiplicity $m$ such that $e_l \ge m/3$. Then all descendants $T$ of $S$ in $\mathcal{T}$ satisfy Wilf's conjecture.
\end{corollary}

\begin{proof}
Let $T$ be a descendant of $S$. Then, by a repeated application of Proposition~\ref{prop child}\:(2) for each generation of children going down from $S$ to $T$, we have  $$e(T) \ge e_l(T) \ge e_l(S) \ge m/3=m(T)/3.$$ It then follows from Theorem~\ref{thm known cases}\:(3) that $T$ satisfies Wilf's conjecture.
\end{proof}

\begin{corollary}\label{cor e} Let $G \in \mathbb{N}^*$. Let $S$ be a numerical semigroup of genus $g \le G$ such that $e \ge m/3+(G-g)$. Then all descendants $T$ of $S$ of genus $g(T) \le G$ satisfy Wilf's conjecture.
\end{corollary}

\begin{proof}
Let $h=g(T)$. Then $g \le h \le G$ by hypothesis, and $T$ is an $(h-g)$th descendant of $S$. Now at each step from $S$ down to $T$ in $\mathcal{T}$, the number of primitive elements diminishes by at most $1$ as stated in Proposition~\ref{prop child}\:(3). Hence $$e(T) \ge e-(h-g) \ge e+g-G \ge m/3+(G-g)+g-G=m(T)/3.$$ It then follows from Theorem~\ref{thm known cases}\:(3) that $T$ satisfies Wilf's conjecture.
\end{proof}

Consequently, when exploring $\mathcal{T}_G$ to probe Wilf's conjecture up to a given maximal genus $G$, the subtree rooted at any numerical semigroup $S$ satisfying Corollary~\ref{cor el} or~\ref{cor e} can be completely cut off from $\mathcal{T}_G$. What remains after this systematic trimming of $\mathcal{T}_G$ is a subtree $\mathcal{T}_G(3)$ all of whose numerical semigroups $S$ satisfy

\begin{enumerate}
\item $e_l(S) < m(S)/3$, \vspace{1mm}
\item $e(S) < m(S)/3 + (G-g(S))$.
\end{enumerate}

Summarizing, to probe Wilf's conjecture up to genus $G$, we only need to test those numerical semigroups $S$ in $\mathcal{T}_G(3)$. This is a significant reduction, as the subtree $\mathcal{T}_G(3)$ turns out to be much smaller than $\mathcal{T}_G$. For instance, for $G=100$, we found that $\mathcal{T}_{100}(3)$ counts approximately $4.5 \times 10^{15}$ nodes, as compared to the full tree $\mathcal{T}_{100}$ counting roughly $n_{75}\times \phi^{25} \approx 3.2 \times 10^{19}$ nodes. (See Section~\ref{sec experiments}.) This level of reduction allowed us to reach one of the main computational results of this paper.

\begin{theorem}\label{thm wilf 100}
    Wilf's conjecture holds for all numerical semigroups of genus $g \le 100$. $\Box$
\end{theorem}
\begin{proof}
    By computer, we constructed the much smaller subtree $\mathcal{T}_G(3) \subset \mathcal{T}_G$ for $G=100$, and checked that all of its nodes  satisfy Wilf's conjecture. The claimed statement follows from Corollary~\ref{cor el} and~\ref{cor e}.
\end{proof}

\subsection{Further trimming} There are ways to further trim $\mathcal{T}_G(3)$ and thus further reduce the number of numerical semigroups of genus $g \le G$ to test for Wilf's conjecture. They may be used whenever the added computational cost remains reasonable. 

Here is a useful instance, exploiting the result that numerical semigroups with embedding dimension $e \le 3$ satisfy Wilf's conjecture~\cite{FroebergGottliebHaeggkvist1987SF-numerical}, as recalled in Theorem~\ref{thm known cases}\:(1).

\begin{proposition}\label{prop 4g/3} Let $S$ be a numerical semigroup such that $c \ge 4g/3$. Then $S$ satisfies Wilf's conjecture.
\end{proposition}
\begin{proof} As $g=c-|L|$, the hypothesis yields $c \ge 4(c-|L|)/3$, i.e. $4|L| \ge c$. If $e \ge 4$ then $e|L| \ge 4|L| \ge c$, so $S$ satisfies Wilf's conjecture. And if $e \le 3$, the same conclusion holds by Theorem~\ref{thm known cases}\:(1).
\end{proof}
\begin{corollary} Let $G \in \mathbb{N}^*$. Let $S$ be a numerical semigroup of genus $g \le G$ such that $|L| \ge G/3$. Then all descendants $T$ of $S$ of genus $g(T) \le G$ satisfy Wilf's conjecture.
\end{corollary}
\begin{proof} Let $T$ be a descendant of $S$ of genus $g(T) \le G$. We have
$$
\begin{aligned}
c(T) &= |L(T)|+g(T) \\
& \ge |L|+g(T) \\
& \ge G/3+g(T) \\
& \ge g(T)/3+g(T) \\
& = 4g(T)/3. 
\end{aligned} 
$$
Proposition~\ref{prop 4g/3} then implies that $T$ satisfies Wilf's conjecture.
\end{proof}

\section{The case $c \in m\N$}\label{sec special}

Given a numerical semigroup $S$, consider the Euclidean division of its conductor $c$ by its multiplicity $m$ with nonpositive remainder: 
\begin{equation}\label{rho}
c= qm-\rho, \quad 0 \le \rho \le m-1.
\end{equation}

As argued in~\cite{Eliahou2023pp-Divsets}, the particular case $c=qm$, i.e. with $\rho=0$, might well be the heart of Wilf's conjecture. Indeed, the proofs of Wilf's conjecture in either case $c \le 3m$~\cite{Eliahou2018JEMS-Wilfs} or $e \ge m/3$~\cite{Eliahou2020EJC-graph} can be significantly shortened in this particular case. Moreover, the first five instances of the rare occurrence $W_0(S) < 0$ all belong to this case~\cite{EliahouFromentin2019SF-misses}.

\begin{definition} A numerical semigroup $S$ is \emph{special} if its multiplicity $m$ divides its conductor $c$.
\end{definition}

For instance, the ordinary numerical semigroup $O_m=\{0\} \cup (m+\N)$ is special since it satisfies $c=m$. The above discussion leads one to think that, if Wilf's conjecture is false, then some counterexamples might well be special. Thus, the special case should be given priority in related research work.

In this context, the validity of Wilf's conjecture has been extended to the case  $e \ge m/4$ for special numerical semigroups. More generally, using the notation in~\eqref{rho}, the result reads as follows~\cite{Eliahou2023pp-Divsets}.

\begin{theorem}\label{thm m/4} Let $S$ be a numerical semigroup satisfying $e \ge m/4$. Then $e|L|\ge c-m+\rho$. Morever, if $S$ is special, then $e|L|\ge c$, i.e. $S$ satisfies Wilf's conjecture.
\end{theorem}

Let us see how we used here this theorem to push the verification of Wilf's conjecture up to genus $G=120$ in the special case $c \in m\N$.

\smallskip
To start with, the method in Section~\ref{sec trim} to trim $\mathcal{T}$ or its bounded version $\mathcal{T}_G$ to get the greatly reduced relevant subtree $\mathcal{T}_G(3)$ in case $e < m/3$ works as well to retain only those numerical semigroups such that $e < m/d$ for some fixed integer $d \ge 3$. This yields a subtree $\mathcal{T}_G(d)$ of $\mathcal{T}_G$ all of whose numerical semigroups $S$ satisfy  \vspace{1mm}

\begin{enumerate}
\item $e_l(S) < m(S)/d$, \vspace{1mm}
\item $e(S) < m(S)/d + (G-g(S))$.
\end{enumerate}

\smallskip
By Theorem~\ref{thm m/4}, the case of relevance to us here is $d=4$. The subtree $\mathcal{T}_G(4)$ may be further trimmed by exploiting the added hypothesis $c=qm$, as explained below. For convenience, we introduce the following definition.

\subsection{Special trimming}{} We start with a condition on a non-special numerical semigroup $S$ ensuring that no descendant of $S$ is special.

\begin{proposition} Let $S$ be a non-special numerical semigroup. Assume further that $S$ has no right primitive element $a \in P \cap R$ such that $a \equiv -1 \bmod m$. Then no descendant $T$ of $S$ is special.
\end{proposition}
\begin{proof} Let $S'=S\setminus \{a\}$ with $a \in P \cap R$ be a child of $S$. As usual, we denote by $m,F,c,P$ the multiplicity, Frobenius number, conductor and primitive elements of $S$, and by $m',F',c',P'$ the corresponding objects for $S'$. First of all, $S \neq O_m$ since $O_m$ is special. Hence $m'=m$ by Proposition~\ref{prop child}. We have $F'=a$ and so $c'=a+1$. Since $a \not\equiv -1 \bmod m$ by hypothesis, and $m'=m$, we have $c' \not \equiv 0 \bmod m'$. Therefore $S'$ is not special. Moreover, no right primitive element of $S'$ can be congruent to $-1$ mod $m'$ since $P'=P \setminus \{a\}$ or $P \setminus \{a\} \sqcup \{a+m\}$ as seen in the proof of Proposition~\ref{prop child}. Hence $S'$ satisfies the same hypotheses as $S$ and we are done by induction on the distance between $S$ and any of its descendant $T$.
\end{proof}

We denote by $\mathcal{T}'_{G}(d)$ the subtree obtained from $\mathcal{T}_{G}(d)$ by the above special trimming.

\subsection{Outcome}

Having constructed the subtree $\mathcal{T}'_{G}(d)$ for $d=4$ and maximal genus $G=120$, we ended up with the following computational result.

\begin{theorem}\label{thm special 120} Wilf's conjecture holds for all special numerical semigroups $S$ of genus $g \le 120$.
\end{theorem}
\begin{proof} For each numerical semigroup $S$ in the subtree $\mathcal{T}'_{120}(4)$, we computed $W(S)=e|L|-c$ and found that $W(S) \ge 0$, i.e. that $S$ satisfies Wilf's conjecture. All numerical semigroups of genus $g \le 120$ outside $\mathcal{T}'_{120}(4)$ satisfy $e \ge m/4$ or are not special. Combining the computational results on $\mathcal{T}'_{120}(4)$ and Theorem~\ref{thm m/4}, we conclude that all special numerical semigroups of genus $g \le 120$ satisfy Wilf's conjecture.
\end{proof}

\section{Computational results}\label{sec experiments}

Our experiments were carried out on the computational platform CALCULCO~\cite{calculco} using an adapted and distributed version of algorithms given in \cite{FromentinHivert2016MC-Exploring}. Source codes are available on GitHub~\cite{github-code}.
These experiments were run for two weeks during summer 2023, when the platform CALCULCO, shared by the whole university, had a much better availability. 
The distributed computation ran in parallel on up to 1500 cores, with a special-purpose finely tuned dynamic balancing of the tasks of the respective cores. Early experiments were conducted using the GAP package \textit{NumericalSgps}~\cite{NumericalSgps1.3.1}.

\subsection{Exploration of $\mathcal{T}_{100}(3)$}

The following table gives the number $t_g$ of numerical semigroups of genus $g$ in the trimmed subtree $\mathcal{T}_{100}(3)$.

\[
\footnotesize
\rowcolors{1}{}{black!05}
\begin{array}{|r|r||r|r||r|r|}
\hline
g&t_g&g&t_g&g&t_g\\
\hline
0 & 1 & 34 & 183\ 029 & 68 & 78\ 371\ 434\ 661 \\
1 & 1 & 35 & 268\ 072 & 69 & 114\ 677\ 728\ 452 \\
2 & 1 & 36 & 392\ 646 & 70 & 167\ 759\ 612\ 028 \\
3 & 1 & 37 & 575\ 237 & 71 & 245\ 327\ 971\ 537 \\
4 & 2 & 38 & 842\ 632 & 72 & 358\ 502\ 883\ 157 \\
5 & 3 & 39 & 1\ 234\ 294 & 73 & 523\ 268\ 737\ 918 \\
6 & 4 & 40 & 1\ 808\ 003 & 74 & 762\ 512\ 542\ 535 \\
7 & 6 & 41 & 2\ 648\ 088 & 75 & 1\ 108\ 797\ 952\ 894 \\
8 & 9 & 42 & 3\ 878\ 863 & 76 & 1\ 608\ 029\ 199\ 893 \\
9 & 13 & 43 & 5\ 681\ 044 & 77 & 2\ 323\ 793\ 898\ 612 \\
10 & 19 & 44 & 8\ 320\ 312 & 78 & 3\ 343\ 540\ 732\ 459 \\
11 & 28 & 45 & 12\ 184\ 995 & 79 & 4\ 786\ 270\ 172\ 173 \\
12 & 41 & 46 & 17\ 844\ 810 & 80 & 6\ 811\ 932\ 935\ 500 \\
13 & 60 & 47 & 26\ 134\ 470 & 81 & 9\ 633\ 271\ 340\ 874 \\
14 & 88 & 48 & 38\ 275\ 824 & 82 & 13\ 524\ 365\ 031\ 892 \\
15 & 129 & 49 & 56\ 052\ 677 & 83 & 18\ 835\ 200\ 708\ 312 \\
16 & 189 & 50 & 82\ 079\ 784 & 84 & 26\ 006\ 592\ 640\ 071 \\
17 & 277 & 51 & 120\ 191\ 188 & 85 & 35\ 586\ 447\ 144\ 420 \\
18 & 406 & 52 & 176\ 010\ 965 & 86 & 48\ 235\ 329\ 094\ 317 \\
19 & 595 & 53 & 257\ 743\ 713 & 87 & 64\ 707\ 333\ 203\ 651 \\
20 & 872 & 54 & 377\ 377\ 331 & 88 & 85\ 854\ 587\ 472\ 809 \\
21 & 1\ 278 & 55 & 552\ 530\ 112 & 89 & 112\ 592\ 214\ 454\ 082 \\
22 & 1\ 870 & 56 & 809\ 003\ 680 & 90 & 145\ 836\ 255\ 324\ 616 \\
23 & 2\ 741 & 57 & 1\ 184\ 568\ 132 & 91 & 186\ 464\ 879\ 487\ 116 \\
24 & 4\ 019 & 58 & 1\ 734\ 367\ 942 & 92 & 234\ 882\ 687\ 403\ 501 \\
25 & 5\ 888 & 59 & 2\ 539\ 101\ 162 & 93 & 290\ 865\ 320\ 646\ 279 \\
26 & 8\ 622 & 60 & 3\ 717\ 160\ 466 & 94 & 353\ 167\ 513\ 519\ 152 \\
27 & 12\ 634 & 61 & 5\ 441\ 979\ 825 & 95 & 419\ 043\ 410\ 131\ 476 \\
28 & 18\ 513 & 62 & 7\ 967\ 290\ 270 & 96 & 483\ 141\ 727\ 918\ 288 \\
29 & 27\ 128 & 63 & 11\ 663\ 422\ 314 & 97 & 534\ 768\ 932\ 735\ 380 \\
30 & 39\ 749 & 64 & 17\ 072\ 801\ 062 & 98 & 557\ 018\ 016\ 635\ 015 \\
31 & 58\ 192 & 65 & 24\ 990\ 316\ 134 & 99 & 522\ 447\ 041\ 258\ 147 \\
32 & 85\ 285 & 66 & 36\ 581\ 421\ 194 & 100 & 389\ 883\ 092\ 218\ 470 \\
33 & 124\ 928 & 67 & 53\ 548\ 048\ 989 &  &  \\
\hline
\end{array}
\]

The number of nodes in $\mathcal{T}_{100}(3)$ is equal to $4\ 554\,895\, 996\, 302\, 538\approx 4.5\times 10^{15}$. We checked that all of these numerical semigroups satisfy Wilf's conjecture, whence Theorem~\ref{thm wilf 100}.

\subsection{Exploration of $\mathcal{T}'_{120}(4)$}
The following table gives the number $t'_g$ of numerical semigroups of genus $g$ in the subtree $\mathcal{T}'_{120}(4)$.

\[
\footnotesize
\rowcolors{1}{}{black!05}
\begin{array}{|r|r||r|r||r|r|}
\hline
g&t'_g&g&t'_g&g&t'_g\\
\hline
0 & 1 & 41 & 200\,122 & 82 & 98\,068\,856\,236 \\
1 & 1 & 42 & 278\,371 & 83 & 136\,816\,899\,688 \\
2 & 1 & 43 & 369\,269 & 84 & 188\,124\,954\,369 \\
3 & 1 & 44 & 510\,693 & 85 & 246\,090\,486\,177 \\
4 & 1 & 45 & 699\,711 & 86 & 336\,614\,411\,642 \\
5 & 2 & 46 & 975\,178 & 87 & 446\,053\,184\,686 \\
6 & 3 & 47 & 1\,342\,072 & 88 & 606\,309\,920\,447 \\
7 & 4 & 48 & 1\,876\,236 & 89 & 801\,094\,605\,082 \\
8 & 5 & 49 & 2\,608\,650 & 90 & 1\,075\,828\,933\,040 \\
9 & 7 & 50 & 3\,458\,914 & 91 & 1\,412\,830\,185\,991 \\
10 & 10 & 51 & 4\,794\,003 & 92 & 1\,734\,561\,883\,001 \\
11 & 14 & 52 & 6\,551\,846 & 93 & 2\,241\,022\,409\,143 \\
12 & 19 & 53 & 9\,147\,280 & 94 & 2\,760\,341\,832\,836 \\
13 & 26 & 54 & 12\,582\,317 & 95 & 3\,516\,881\,984\,622 \\
14 & 36 & 55 & 17\,614\,571 & 96 & 4\,299\,648\,368\,842 \\
15 & 49 & 56 & 24\,483\,065 & 97 & 5\,388\,292\,247\,874 \\
16 & 67 & 57 & 32\,457\,488 & 98 & 6\,544\,140\,003\,541 \\
17 & 93 & 58 & 45\,063\,305 & 99 & 7\,117\,478\,807\,043 \\
18 & 128 & 59 & 61\,488\,392 & 100 & 8\,451\,858\,568\,006 \\
19 & 177 & 60 & 85\,953\,600 & 101 & 9\,261\,220\,874\,872 \\
20 & 245 & 61 & 118\,226\,772 & 102 & 10\,843\,978\,899\,677 \\
21 & 340 & 62 & 165\,696\,926 & 103 & 11\,904\,592\,718\,137 \\
22 & 455 & 63 & 230\,209\,288 & 104 & 13\,736\,474\,753\,272 \\
23 & 624 & 64 & 305\,247\,236 & 105 & 15\,172\,362\,267\,910 \\
24 & 863 & 65 & 424\,326\,522 & 106 & 13\,564\,852\,076\,075 \\
25 & 1\,194 & 66 & 578\,281\,772 & 107 & 14\,519\,416\,932\,134 \\
26 & 1\,647 & 67 & 809\,156\,311 & 108 & 13\,448\,628\,571\,779 \\
27 & 2\,286 & 68 & 1\,112\,860\,410 & 109 & 14\,268\,658\,755\,506 \\
28 & 3\,180 & 69 & 1\,561\,100\,560 & 110 & 13\,799\,851\,102\,125 \\
29 & 4\,234 & 70 & 2\,168\,306\,879 & 111 & 14\,508\,263\,526\,352 \\
30 & 5\,823 & 71 & 2\,876\,214\,827 & 112 & 14\,534\,198\,939\,692 \\
31 & 8\,035 & 72 & 4\,002\,364\,427 & 113 & 11\,099\,577\,260\,819 \\
32 & 11\,135 & 73 & 5\,449\,537\,905 & 114 & 10\,699\,792\,288\,594 \\
33 & 15\,341 & 74 & 7\,630\,005\,823 & 115 & 9\,284\,396\,042\,115 \\
34 & 21\,369 & 75 & 10\,492\,890\,135 & 116 & 8\,422\,462\,308\,726 \\
35 & 29\,722 & 76 & 14\,726\,118\,585 & 117 & 6\,954\,667\,530\,694 \\
36 & 39\,491 & 77 & 20\,444\,255\,810 & 118 & 5\,013\,091\,736\,917 \\
37 & 54\,511 & 78 & 27\,121\,425\,859 & 119 & 2\,599\,964\,149\,312 \\
38 & 74\,910 & 79 & 37\,736\,682\,161 & 120 & 289\,298\,823\,487 \\
39 & 104\,183 & 80 & 51\,267\,633\,069 & &  \\
40 & 143\,431 & 81 & 71\,625\,262\,707 &  & \\
\hline
\end{array}
\]

The number of nodes in $\mathcal{T}'_{120}(4)$ is equal to $261\,588\,966\,883\,192\approx 2.6\times 10^{14}$. We checked that all of these numerical semigroups  satisfy Wilf's conjecture, whence Theorem~\ref{thm special 120}.

\subsection{Growth rates}
Recall that $t_g$ and $t'_g$ denote the number of numerical semigroups of genus $g$ in $\mathcal{T}_{100}(3)$ and $\mathcal{T}'_{120}(4)$, respectively. Recall also that the growth rate of $n_g$ tends to $\phi=\frac{1+\sqrt5}2\approx 1.62$~\cite{Zhai}.

As illustrated by the following figure, the growth rate of $t_g$ seems to stabilize around $1.46$ for $g\in[10,70]$.

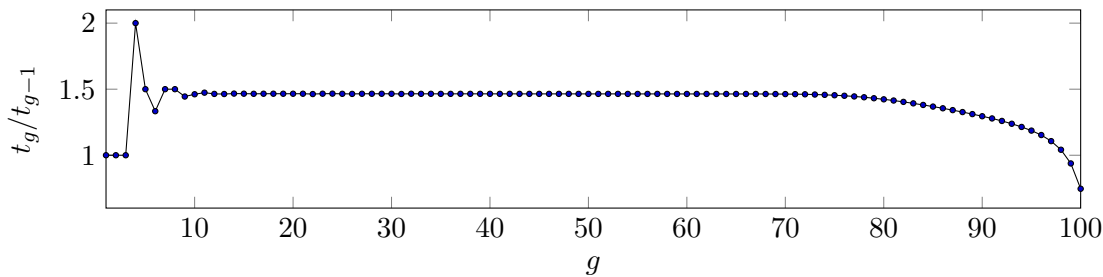
\begin{center}
\small
\begin{tikzpicture}
\begin{axis}[xmin=1,xmax=100,ymin=0.6,ymax=2.1,height=11em,width=0.9\textwidth,xlabel=$g$,ylabel=$t_g/t_{g-1}$]
\addplot+[color=black,mark size=1pt]
table[meta=r]{ratio_100.txt};
\end{axis}
\end{tikzpicture}
\end{center}

The growth rate of $t'_g$ is a little more chaotic but seems to have the same behavior, as illustrated by the following figure.

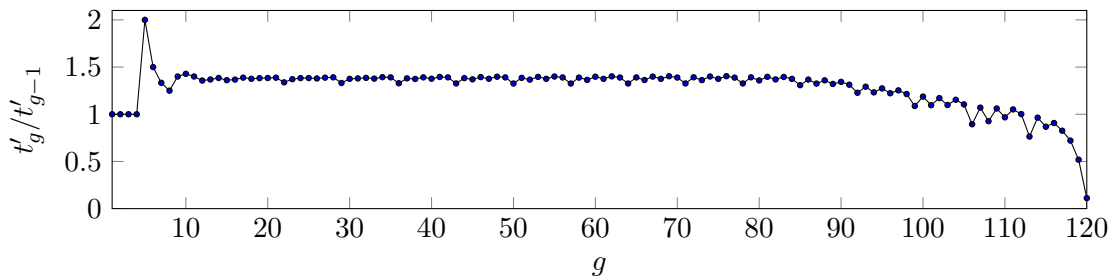
\begin{center}
\begin{tikzpicture}
\small
\begin{axis}[xmin=1,xmax=120,ymin=0,ymax=2.1,height=11em,width=0.9\textwidth,xlabel=$g$,ylabel=$t'_g/t'_{g-1}$]
\addplot+[color=black,mark size=1pt]
table[meta=r]{ratio_120.txt};
\end{axis}
\end{tikzpicture}
\end{center}

\subsection{New values of $n_g$}
We took advantage of our distributed version of the algorithms in \cite{FromentinHivert2016MC-Exploring} to compute the values of $n_g$ for all $g\leq 75$. See Table~\ref{T:75}. The total number of numerical semigroups of genus $g \leq 75$ is equal to $46\,844\,766\,834\,597\,649\approx 4.6\times 10^{16}$. 

\begin{table}[h]
\[
\footnotesize
\rowcolors{1}{}{black!05}
\begin{array}{|r|r||r|r||r|r|}
\hline
g&n_g&g&n_g&g&n_g\\
\hline
0 & 1 & 26 & 770\,832 & 52 & 266\,815\,155\,103 \\
1 & 1 & 27 & 1\,270\,267 & 53 & 433\,317\,458\,741 \\
2 & 2 & 28 & 2\,091\,030 & 54 & 703\,569\,992\,121 \\
3 & 4 & 29 & 3\,437\,839 & 55 & 1\,142\,140\,736\,859 \\
4 & 7 & 30 & 5\,646\,773 & 56 & 1\,853\,737\,832\,107 \\
5 & 12 & 31 & 9\,266\,788 & 57 & 3\,008\,140\,981\,820 \\
6 & 23 & 32 & 15\,195\,070 & 58 & 4\,880\,606\,790\,010 \\
7 & 39 & 33 & 24\,896\,206 & 59 & 7\,917\,344\,087\,695 \\
8 & 67 & 34 & 40\,761\,087 & 60 & 12\,841\,603\,251\,351 \\
9 & 118 & 35 & 66\,687\,201 & 61 & 20\,825\,558\,002\,053 \\
10 & 204 & 36 & 109\,032\,500 & 62 & 33\,768\,763\,536\,686 \\
11 & 343 & 37 & 178\,158\,289 & 63 & 54\,749\,244\,915\,730 \\
12 & 592 & 38 & 290\,939\,807 & 64 & 88\,754\,191\,073\,328 \\
13 & 1\,001 & 39 & 474\,851\,445 & 65 & 143\,863\,484\,925\,550 \\
14 & 1\,693 & 40 & 774\,614\,284 & 66 & 233\,166\,577\,125\,714 \\
15 & 2\,857 & 41 & 1\,262\,992\,840 & 67 & 377\,866\,907\,506\,273 \\
16 & 4\,806 & 42 & 2\,058\,356\,522 & 68 & 612\,309\,308\,257\,800 \\
17 & 8\,045 & 43 & 3\,353\,191\,846 & 69 & 992\,121\,118\,414\,851 \\
18 & 13\,467 & 44 & 5\,460\,401\,576 & 70 & 1\,607\,394\,814\,170\,158 \\
19 & 22\,464 & 45 & 8\,888\,486\,816 & 71 & 2\,604\,033\,182\,682\,582 \\
20 & 37\,396 & 46 & 14\,463\,633\,648 & 72 & 4\,218\,309\,716\,540\,814 \\
21 & 62\,194 & 47 & 23\,527\,845\,502 & 73 & \bf 6\,832\,823\,876\,813\,577 \\
22 & 103\,246 & 48 & 38\,260\,496\,374 & 74 & \bf 11\,067\,092\,660\,179\,522 \\
23 & 170\,963 & 49 & 62\,200\,036\,752 & 75 & \bf 17\,924\,213\,336\,425\,401 \\
24 & 282\,828 & 50 & 101\,090\,300\,128 &  &  \\
25 & 467\,224 & 51 & 164\,253\,200\,784 &  &  \\
\hline
\end{array}
\]
\caption{\label{T:75} Number $n_g$ of numerical semigroups of genus~$g\leq 75$. Previously unknown values are in bold font.}
\end{table}


\bibliographystyle{plain}
\bibliography{refs.bib}

\begin{thebibliography}{10}

\bibitem{calculco}
{CALCULCO} platform.
\newblock \url{https://www-calculco.univ-littoral.fr}, 2023.

\bibitem{Bras-Amoros2008SF-Fibonacci}
Maria Bras-Amor\'{o}s.
\newblock Fibonacci-like behavior of the number of numerical semigroups of a
  given genus.
\newblock {\em Semigroup Forum}, 76(2):379--384, 2008.

\bibitem{Bras-AmorosRodriguez2021-New}
Maria Bras-Amor\'{o}s and C\'{e}sar~Mar\'{\i}n Rodr\'{\i}guez.
\newblock New {E}liahou semigroups and verification of the {W}ilf conjecture
  for genus up to 65.
\newblock 12898:17--27, [2021] \copyright 2021.

\bibitem{Bras-Amoros2023pp-seeds}
Maria Bras-Amorós.
\newblock On the seeds and the great-grandchildren of a numerical semigroup,
  2023.
\newblock To appear in Mathematics of Computation.

\bibitem{BrunsGarcia-SanchezONeillWilburne2020IJAC-Wilfs}
Winfried Bruns, Pedro Garc\'{\i}a-S\'{a}nchez, Christopher O'Neill, and Dane
  Wilburne.
\newblock Wilf's conjecture in fixed multiplicity.
\newblock {\em Internat. J. Algebra Comput.}, 30(4):861--882, 2020.

\bibitem{github-code}
M.~Delgado, S.~Eliahou, and J.~Fromentin.
\newblock Github - trim\_tree\_semigroups.
\newblock \url{https://github.com/jfromentin/trim_tree_semigroups}, 2023.

\bibitem{NumericalSgps1.3.1}
M.~Delgado, P.~A. Garcia-Sanchez, and J.~Morais.
\newblock {NumericalSgps}, a package for numerical semigroups, {V}ersion 1.3.1,
  July 2022.
\newblock Refereed GAP package.

\bibitem{Delgado2018MZ-question}
Manuel Delgado.
\newblock On a question of {E}liahou and a conjecture of {W}ilf.
\newblock {\em Math. Z.}, 288(1-2):595--627, 2018.

\bibitem{Delgado2019ae-Trimming}
Manuel Delgado.
\newblock Trimming the numerical semigroups tree to probe {W}ilf's conjecture
  to higher genus, 2019.
\newblock 40 pages.

\bibitem{Delgado2020-survey}
Manuel Delgado.
\newblock Conjecture of {W}ilf: A survey.
\newblock In V.~Barucci, S.~Chapman, M.~D'Anna, and R.~Fröberg, editors, {\em
  Numerical Semigroups}, volume~40 of {\em Springer INdAM Ser.}, pages 39--62.
  Springer, Cham, 2020.

\bibitem{DobbsMatthews2006}
David~E. Dobbs and Gretchen~L. Matthews.
\newblock On a question of {W}ilf concerning numerical semigroups.
\newblock In {\em Focus on commutative rings research}, pages 193--202. Nova
  Sci. Publ., New York, 2006.

\bibitem{Eliahou2018JEMS-Wilfs}
Shalom Eliahou.
\newblock Wilf’s conjecture and {M}acaulay’s theorem.
\newblock {\em J. Eur. Math. Soc. (JEMS)}, 20(9):2105--2129, 2018.

\bibitem{Eliahou2020EJC-graph}
Shalom Eliahou.
\newblock A graph-theoretic approach to {W}ilf ’s conjecture.
\newblock {\em Electron. J. Combin.}, 27 (2):Article No P2.15, 31 pages, 2020.

\bibitem{Eliahou2023pp-Divsets}
Shalom Eliahou.
\newblock {Divsets, numerical semigroups and Wilf's conjecture}.
\newblock \url{https://hal.science/hal-04234167}, October 2023.

\bibitem{EliahouFromentin2019SF-misses}
Shalom Eliahou and Jean Fromentin.
\newblock Near-misses in {W}ilf's conjecture.
\newblock {\em Semigroup Forum}, 98(2):285--298, 2019.

\bibitem{EliahouMarin-Aragon2021CiA-numerical}
Shalom Eliahou and Daniel Mar\'{\i}n-Arag\'{o}n.
\newblock On numerical semigroups with at most 12 left elements.
\newblock {\em Comm. Algebra}, 49(6):2402--2422, 2021.

\bibitem{FroebergGottliebHaeggkvist1987SF-numerical}
R.~Fr\"oberg, C.~Gottlieb, and R.~H\"aggkvist.
\newblock On numerical semigroups.
\newblock {\em Semigroup Forum}, 35(1):63--83, 1987.

\bibitem{FromentinHivert2016MC-Exploring}
Jean Fromentin and Florent Hivert.
\newblock Exploring the tree of numerical semigroups.
\newblock {\em Math. Comp.}, 85(301):2553--2568, 2016.

\bibitem{Kaplan2012JPAA-Counting}
Nathan Kaplan.
\newblock Counting numerical semigroups by genus and some cases of a question
  of {W}ilf.
\newblock {\em J. Pure Appl. Algebra}, 216(5):1016--1032, 2012.

\bibitem{KliemStump2022DCG-new}
Jonathan Kliem and Christian Stump.
\newblock A new face iterator for polyhedra and for more general finite locally
  branched lattices.
\newblock {\em Discrete Comput. Geom.}, 67(4):1147--1173, 2022.

\bibitem{MoscarielloSammartano2015MZ-conjecture}
Alessio Moscariello and Alessio Sammartano.
\newblock On a conjecture by {W}ilf about the {F}robenius number.
\newblock {\em Math. Z.}, 280(1-2):47--53, 2015.

\bibitem{Ramirez-Alfonsin2005Book-Diophantine}
J.~L. Ramírez-Alfonsín.
\newblock {\em The {D}iophantine {F}robenius problem}, volume~30 of {\em Oxford
  Lecture Series in Mathematics and its Applications}.
\newblock Oxford University Press, Oxford, 2005.

\bibitem{RosalesGarcia2009Book-Numerical}
J.~C. Rosales and P.~A. García~Sánchez.
\newblock {\em Numerical semigroups}, volume~20 of {\em Developments in
  Mathematics}.
\newblock Springer, New York, 2009.

\bibitem{Sammartano2012SF-Numerical}
Alessio Sammartano.
\newblock Numerical semigroups with large embedding dimension satisfy {W}ilf's
  conjecture.
\newblock {\em Semigroup Forum}, 85(3):439--447, 2012.

\bibitem{Selmer1977JRAM-linear}
Ernst~S. Selmer.
\newblock On the linear {D}iophantine problem of {F}robenius.
\newblock {\em J. Reine Angew. Math.}, 293(294):1--17, 1977.

\bibitem{Sylvester1882AJM-Subvariants}
J.~J. Sylvester.
\newblock On {S}ubinvariants, i.e. {S}emi-{I}nvariants to {B}inary {Q}uantics
  of an {U}nlimited {O}rder.
\newblock {\em Amer. J. Math.}, 5(1-4):79--136, 1882.

\bibitem{Wilf1978AMM-circle}
Herbert~S. Wilf.
\newblock A circle-of-lights algorithm for the ``money-changing problem''.
\newblock {\em Amer. Math. Monthly}, 85(7):562--565, 1978.

\bibitem{Zhai}
Alex Zhai.
\newblock Fibonacci-like growth of numerical semigroups of a given genus.
\newblock {\em Semigroup Forum}, 86(3):634--662, 2013.

\end{thebibliography}
\bigskip

\noindent
{\small
\textbf{Authors addresses}

\medskip

\noindent
$\bullet$ Manuel {\sc Delgado},

\noindent
{CMUP--Centro de Matemática da Universidade do Porto, Departamento de Matemática, Faculdade de Ciências, Universidade do Porto, Rua do Campo Alegre s/n, 4169– 007 Porto, Portugal\\
\email{mdelgado@fc.up.pt}

\medskip

\noindent
$\bullet$ Shalom {\sc Eliahou}\textsuperscript{a,b},

\noindent
\textsuperscript{a}Univ. Littoral C\^ote d'Opale, UR 2597 - LMPA - Laboratoire de Math\'ematiques Pures et Appliqu\'ees Joseph Liouville, F-62100 Calais, France\\
\textsuperscript{b}CNRS, FR2037, France\\
\email{eliahou@univ-littoral.fr}

\medskip

\noindent
$\bullet$ Jean {\sc Fromentin}\textsuperscript{a,b},

\noindent
\textsuperscript{a}Univ. Littoral C\^ote d'Opale, UR 2597 - LMPA - Laboratoire de Math\'ematiques Pures et Appliqu\'ees Joseph Liouville, F-62100 Calais, France\\
\textsuperscript{b}CNRS, FR2037, France\\
\email{fromentin@univ-littoral.fr}
}

\end{document}